\documentclass[11pt]{amsart}
\usepackage{enumitem}
\usepackage[colorlinks=true]{hyperref}

\newtheorem{thm}{Theorem}[section]
\newtheorem{lem}[thm]{Lemma}
\newtheorem{prop}[thm]{Proposition}

\theoremstyle{definition}
\newtheorem{defn}[thm]{Definition}

\title{ Semiring arising as Lattice of Groupsemirings }
\author{A. R. Rajan}
\address{Department of Mathematics, University of Kerala, Kariavattom, Kerala, India}
\email{arrunivker@gmail.com}
\author{S. Sheena}
\address{Department of Mathematics, College of Engineering, Thiruvananthapuram, India}
\email{sheena\_s@cet.ac.in}
\author{C. S. Preenu}
\address{Department of Mathematics, University College Thiruvananthapuram, Kerala, India}
\email{cspreenu@gmail.com}

\begin{document}
\begin{abstract} 
Much study has been done on semigroups which are unions of groups. There are several ways in which a union of groups can be made into a semigroup in which each of the 
component groups arises as subgroups of the constructed semigroup.
An important  class of such unions is a semilattice of groups. Group semirings are 
semirings $(G,+,\cdot )$ where $(G,\cdot )$ is  a group and $(G,+)$ is a left zero semigroup.
We consider construction of semirings from classes of group semirings 
$\{G_\alpha :\alpha\in D \}$ indexed by a distributive lattice $D$. 
It is shown that if $S=\cup\{G_\alpha \}$ is a strong distributive lattice of group semirings $G_\alpha$ then the multiplicative semigroup $(S,\cdot)$ of the semiring $(S,+,\cdot)$ is a
Clifford semigroup and the additive semigroup $(S,+)$ is a left normal band. Further in  this case all the groups $G_\alpha $ are mutually isomorphic. 	
\end{abstract}  

\maketitle

\section{Introduction }

Semirings in which the multiplicative semigroup belong to some special
classes of semigroups have been studied
by several people. For example see \cite{sen1993maximal, bhuniya2010distributive} etc.     
Semilattice of groups and srtrong semilattice of groups are well established concepts in structure theory of semigroups. 
Similar constructions in semirings have been considered by Ghosh\cite{ghosh1999characterization}, Bandelt and Petrich \cite{bandelt1982subdirect} etc. where the semilattice is replaced by a distributive lattice. 
A union $S=\cup\{S_\alpha:\alpha\in D \}$ of semirings $S_\alpha$ where $D$ is a distributive lattice is said to be a distributive lattice of semirings if addition and multiplication in $S$ have the following properties.
For $x\in S_\alpha$ and $y\in S_\beta$
\[x+y\in S_{\alpha +\beta}\text{ and }x\cdot y\in S_{\alpha\beta}\]
where $+$ and $\cdot$ denote the join and meet in the distributive lattice $D$.

The above union $S=\cup\{S_\alpha:\alpha\in D \}$ is said to be a strong distributive lattice of semirings if there are families of homomorphisms of semirings 
$\{\phi_{\alpha ,\beta}:S_\alpha\to S_\beta \text{ for }\beta\le\alpha \}$
and $\{\psi_{\beta ,\alpha}:S_\beta\to S_\alpha \text{ for }\beta\le\alpha \}$
with certain properties such that 

\[x+y=x\psi_{\alpha,\alpha+\beta}+y\psi_{\beta ,\alpha+\beta}\]
and
\[x\cdot y=x\phi_{\alpha,\alpha\beta}\cdot y\phi_{\beta ,\alpha\beta}\]
where $x\in S_\alpha$ and $y\in S_\beta$.

Group semirings are semirings $(G,+,\cdot)$ where $(G,\cdot)$ is a group and 
$+$ is defined by $x+y=x$ for all $x,y\in G$. Here we consider a  distributive lattice $D$ of group semirings $\{ G_\alpha :\alpha\in D\}$ and obtain necessary and sufficient condition 
for a  distributive lattice $D$ of group semirings to be a strong distributive lattice of these group semirings. Further we show that
if $S$ is a strong distributive lattice  of group semirings $\{G_\alpha:\alpha\in D \}$ 
then the groups $G_\alpha$'s are mutually isomorphic and the structure homomorphisms 
$\phi_{\alpha ,\beta}$ and $\psi_{\beta ,\alpha}$ are mutually inverse isomorphisms.
Thus one set of structure homomorphisms  $\{\phi_{\alpha ,\beta}\}$ or 
$\{ \psi_{\beta ,\alpha}\}$ is sufficient to describe the construction. Further we show that
when $S$ is a strong distributive lattice  of group semirings $\{G_\alpha:\alpha\in D \}$ 
the multiplicative semigroup $(S,\cdot)$ is a Clifford semigroup and the additive semigroup
$(S,+)$ is a normal band. 

\section{Preliminaries}

Here we provide the definition and elementary properties relating to
regular semigroups and semirings. The main references are \cite{clifford1961algebraic}, \cite{golan2013semirings} \cite{howie1995fundamentals} and \cite{rajan2011idempotents}.

\subsection{Union of Groups }

We follow the notations and terminology of \cite{clifford1961algebraic} and \cite{howie1995fundamentals} for
concepts in semigroup theory.

A semigroup $S$ is said to be a semilattice union of groups if
there is a semilattice $\Lambda$ and groups $\{G_\alpha:\alpha\in \Lambda \}$
such that 
\[S=\cup_{\alpha\in \Lambda }G_\alpha \]
and
\[G_\alpha G_\beta\subseteq G_{\alpha\beta} \]
for all $\alpha,\beta\in\Lambda$.

A strong semilattice of groups is defined as follows.

\begin{defn}\cite{howie1995fundamentals}\label{strongsldefn}   
Let $\Lambda$ be a semilattice and 
\[S=\cup_{\alpha\in \Lambda }G_\alpha \]
be a semilattice of groups where each $G_\alpha$ is a  group.
Then $S$ is said to be a \emph{strong semilattice of groups} if there is a
family of homomorphisms $\{\phi_{\alpha ,\beta}:G_\alpha\to G_\beta :\beta\le\alpha\}$ such that
\begin{enumerate}[label=(\roman*)]
\item \label{one} $\phi_{\alpha ,\alpha}:G_\alpha\to G_\alpha$ is identity map.
\item\label{two}  $\phi_{\alpha ,\beta}\phi_{\beta ,\gamma}=\phi_{\alpha ,\gamma}$ whenever 
$\gamma\le\beta\le\alpha .$
\item If $x\in G_\alpha$ and $y\in G_\beta$ then
\[xy=x\phi_{\alpha ,\alpha\beta}y\phi_{\beta ,\alpha\beta}.\]
\end{enumerate} 

\end{defn}

It may be observed that if $\{G_\alpha :\alpha\in \Lambda \}$
be a family  of groups indexed by a semilattice $\Lambda$ and
 $\{\phi_{\alpha ,\beta}:G_\alpha\to G_\beta :\beta\le\alpha\}$ 
is a family of homomorphisms satisfying \ref{one} and \ref{two} above then  
\[S=\cup_{\alpha\in \Lambda }G_\alpha \]
is a semigroup with product defined by
\[xy=x\phi_{\alpha ,\alpha\beta}y\phi_{\beta ,\alpha\beta}\]
for all $x\in G_\alpha$ and $y\in G_\beta$. 
Thus we can construct a strong semilattice of the groups.

\section{Distributive Lattice of Semirings}

We assume familiarity with the notations and terminology of semirings as in 
\cite{golan2013semirings},\cite{narasimha2022extended} and \cite{sen1993maximal}.
A semiring $S$ is a structure $(S,+,\cdot)$ where $(S,+)$ and $(S,\cdot)$ are semigroups and
\[ x(y+z)=xy+xz\text{ and }(x+y)z=xz+yz\] 
for all $x,y,z\in S$ where $xy $ stands for $x\cdot y$.

A distributive lattice is denoted by $(D,+,\cdot)$ where $+$ stands for the join and $\cdot$ stands for meet. A distributive lattice of semirings is  a union
$\cup\{ S_\alpha:\alpha\in D\}$ of semirings such that 
for $x\in S_\alpha$ and $y\in S_\beta$
\[x+y\in S_{\alpha +\beta}\text{ and }x\cdot y\in S_{\alpha\beta}\]
where $+$ and $\cdot$ denote the join and meet in the distributive lattice $D$.

A strong distributive lattice of semirings is defined as follows.
\begin{defn}\cite{bandelt1982subdirect}\label{strongslgdefn}   
Let $D=(D,+,\cdot )$ be a distributive lattice and 
\[S=\cup_ {\alpha\in D}S_\alpha \]
be a distributive lattice of semirings $\{S_\alpha :\alpha\in D\}$.
Then $S$ is said to be a \emph{strong distributive lattice of semirings} if there
are two families of  homomorphisms 
$\{\phi_{\alpha ,\beta}:S_\alpha\to S_\beta :\beta\le\alpha\}$  and
$\{\psi_{\beta,\alpha}:S_\beta\to S_\alpha :\beta\le\alpha\}$ such that
\begin{enumerate}
\item[(i)] $\phi_{\alpha ,\alpha}:S_\alpha\to S_\alpha$ is identity map.
\item[(ii)]  $\phi_{\alpha ,\beta}\phi_{\beta ,\gamma}=\phi_{\alpha ,\gamma}$ whenever 
$\gamma\le\beta\le\alpha .$
\item[(i)'] $\psi_{\alpha ,\alpha}:S_\alpha\to S_\alpha$ is identity map.
\item[(ii)']  $\psi_{\gamma ,\beta}\psi_{\beta ,\alpha}=\psi_{\gamma ,\alpha}$ whenever 
$\gamma\le\beta\le\alpha .$
\item[(iii)] If $\alpha +\beta\le\delta$ then
$\phi_{\alpha ,\alpha\beta}\psi_{\alpha\beta ,\beta}=\psi_{\alpha ,\delta}\phi_{\delta ,\beta}$.\end{enumerate} 
Further for $x\in S_\alpha$ and $y\in S_\beta$
\[ xy=x\phi_{\alpha ,\alpha\beta}y\phi_{\beta ,\alpha\beta}\]
and
\[ x+y=x\psi_{\alpha ,\alpha+\beta}+y\psi_{\beta ,\alpha +\beta}.\]
\end{defn}

Now we show that every group can be realised as a semiring by keeping the multiplication as the group operation and defining addition in such a way that the
 resulting semigroup is a  left zero semigroup or  a right zero semigroup.
\begin{prop}
Let $G$ be a group.For $x,y\in G$ define 
\[x+y=x\]
for all $x,y\in G$.
Then $(G,+,\cdot )$ is a semiring where $\cdot$ is the multiplication in the group $G$.\qed 
\end{prop}

We call such semirings as group semirings.
For $x,y\in G$ we may define 
\[x+y=y\]
for all $x,y\in G$.
Then also $(G,+,\cdot )$ is a semiring. Here the semigroup $(G,+)$ is a right zero semigroup.

In the following by a group semiring we mean a group $G$ with addition defined by
\[x+y=x\]
for all $x,y\in G$.

\subsection{Distributive Lattice of Group Semirings}

Let $(D,+,\cdot)$ be a distributive lattice and $\{G_\alpha ,\alpha\in D\}$ be a family of group semirings.
Then
\[S=\cup_{\alpha\in D} G_\alpha\]
is said to be a distributive lattice of group semirings $G_\alpha$ if 
for $x\in G_\alpha$ and $y\in G_\beta$
\[x+y\in G_{\alpha +\beta}\text{ and }x\cdot y\in G_{\alpha\beta}.\]
 
 Also $S=\cup_{\alpha\in D} G_\alpha$ is said to be a
 strong  distributive lattice of group semirings $G_\alpha$ if 
there are two families of homomorphisms
 $\{ \phi_{\alpha ,\beta}:G_\alpha\to G_\beta \text{ for }\beta\le \alpha\}$
and 
 $\{ \psi_{\beta ,\alpha}:G_\beta\to G_\alpha \text{ for }\beta\le \alpha\}$ with the above defined properties,
such that for $x\in G_\alpha$ and $y\in G_\beta$
\[ x+y=x\psi_{\alpha ,\alpha+\beta}+y\psi_{\beta ,\alpha +\beta}=x\psi_{\alpha,\alpha+\beta}\]
and
\[ xy=x\phi_{\alpha ,\alpha\beta}y\phi_{\beta ,\alpha\beta}.\]

First we look into some of the properties of $S$ and  these homorphisms $\phi_{\alpha ,\beta}$ and
$\psi_{\beta ,\alpha}$.
 
\begin{thm}\label{stronglatticethm}   
Let $(D,+,\cdot )$ be a distributive lattice and  $S=\cup\{G_\alpha :\alpha\in D \} $ 
be a strong distributive lattice  of groupsemirings $G_\alpha$ with connecting homomorphisms
 $\{ \phi_{\alpha ,\beta}:G_\alpha\to G_\beta \text{ for }\beta\le \alpha\}$
and 
 $\{ \psi_{\beta ,\alpha}:G_\beta\to G_\alpha \text{ for }\beta\le \alpha\}$.
 
 Then we have  the following. Here for each $\alpha\in D$ we denote by $e_\alpha$ the
 identity of the group $G_\alpha$.
\begin{enumerate}[label=\upshape{(\roman*)}]
\item \label{stl1}If $\beta\le \alpha$ in $D$ then $e_\alpha e_\beta =e_\beta$ and
$e_\alpha +e_\beta =e_\alpha$.
\item \label{stl2}If $x\in G_\alpha$ and $y,y'\in G_\beta$ for some $\alpha,\beta\in D$ then
$x+y=x+y'$.
\item \label{stl3}For $\beta\le\alpha$
\[x\phi_{\alpha ,\beta}=xe_\beta\text{ and }y\psi_{\beta,\alpha}=y+e_\alpha\]
for $x\in G_\alpha$ and $y\in G_\beta$.
 \item \label{stl4}If $\gamma\le\beta\le\alpha$ in $D$ then
$\phi_{\alpha ,\beta}\phi_{\beta,\gamma}=\phi_{\alpha,\gamma}$.
\item \label{stl5}Let $\beta\le\alpha$ in $D$. Then $\phi_{\alpha ,\beta}:G_\alpha\to G_\beta$ is an isomorphism and $\psi_{\beta ,\alpha}=(\phi_{\alpha ,\beta})^{-1}$.
\item \label{stl6}If $\alpha+\beta\le\delta$ and $\gamma\le\alpha\beta$ then
$\phi_{\alpha ,\gamma}\psi_{\gamma ,\beta}=\psi_{\alpha ,\delta}\phi_{\delta ,\beta}$.
\end{enumerate} 

\end{thm}

\begin{proof}
From the description of sum and product in strong distributive lattice of group semirings we have for $\beta\le\alpha$
\begin{eqnarray*}
 e_\alpha e_\beta &=& e_\alpha\phi_{\alpha,\alpha\beta}e_\beta\phi_{\beta,\alpha\beta}\\
 &=&e_\alpha\phi_{\alpha,\beta}e_\beta\phi_{\beta,\beta}\text{ since }\alpha\beta=\beta\\
  &=&e_\beta e_\beta\text{ since }\phi_{\alpha,\beta}\text{ is a homomorphism}\\
 &=& e_\beta .
 \end{eqnarray*} 
Similarly we can show that $e_\alpha +e_\beta =e_\alpha$. This proves \ref{stl1}.

Let $x\in G_\alpha$ and $y,y'\in G_\beta$ for some $\alpha,\beta\in D$. Since $S$ is  a strong distributive lattice of group semirings we have
$x+y=x\psi_{\alpha,\alpha+\beta}=x+y'$. Hence \ref{stl2}.

Now let $\beta\le\alpha $ and $x\in G_\alpha$.
Then  
\begin{align*}
xe_\beta &= x\phi_{\alpha,\beta}e_\beta\phi_{\beta,\beta}&\text{ since }\alpha\beta=\beta\\
&= x\phi_{\alpha,\beta}e_\beta &\text{ since }\phi_{\beta,\beta}\text{ is identity map}\\
& = x\phi_{\alpha,\beta}. 
\end{align*}

Similarly we can see that for $y\in G_\beta$
\[ y+e_\alpha =y\psi_{\beta,\alpha }.\]
This proves \ref{stl3}.\\


Now consider $\beta\le\alpha$ and $x\in G_\alpha$. Then
\[ x\phi_{\alpha,\beta}\psi_{\beta,\alpha}=xe_\beta +e_\alpha .\]
Now by \ref{stl2} above
 \begin{align*}
xe_\beta +e_\alpha &=xe_\beta +x &\text{ since } e_\alpha ,x\in G_\alpha\\
&= xe_\beta +xe_\alpha &\text{ since }x=xe_\alpha\\
&= x(e_\beta +e_\alpha )\\
&= xe_\alpha &\text{ since }e_\alpha+e_\beta=e_\alpha\\
& = x.
 \end{align*}

Similarly we can see that for $y\in G_\beta$
\[ y\psi_{\beta,\alpha}\phi_{\alpha,\beta}=(y+e_\alpha )e_\beta =y .\]
Thus $\phi_{\alpha,\beta} $ and $\psi_{\beta,\alpha}$ are mutually inverse isomorphisms.
Hence \ref{stl5} holds.

Now let $\alpha+\beta\le\delta$ and $\gamma\le\alpha\beta$. Then
$\gamma\le \alpha\le\delta$ and $\gamma\le\beta\le \delta$. So
 
\begin{align*}
  \phi_{\delta ,\gamma}&=\phi_{\delta,\alpha}\phi_{\alpha,\gamma}&\text{ by \ref{stl4} above}\\
  &= \phi_{\delta,\beta}\phi_{\beta,\gamma}&\text{ again by \ref{stl4}}
\end{align*}

So we have 
\begin{eqnarray*}
\phi_{\alpha ,\gamma} &=& (\phi_{\delta,\alpha})^{-1}\phi_{\delta,\gamma}\\
&=& (\phi_{\delta,\alpha})^{-1}\phi_{\delta,\beta}\phi_{\beta,\gamma}\\
&=& \psi_{\alpha,\delta}\phi_{\delta,\beta}(\psi_{\gamma.\beta})^{-1}.
\end{eqnarray*}
Therefore
$\phi_{\alpha ,\gamma}\psi_{\gamma ,\beta}=\psi_{\alpha ,\delta}\phi_{\delta ,\beta}$.

\end{proof}

We now provide the construction of strong distributive lattice of group semirings starting
the a collection $\{G_\alpha:\alpha\in D \}$ where $D$ is a distributive lattice.
We see that unlike the description in \cite{bandelt1982subdirect} and \cite{ghosh1999characterization} only one family of homorphisms is suffient for the description. Moreover it turns out that all the groups $G_\alpha$ are isomorphic to each other.

\begin{thm}\label{sdl}
Let $(D,+,\cdot)$ be a distributive lattice and $\{G_\alpha:\alpha\in D \}$ be a family of group semirings. Let 
$\{ \phi_{\alpha ,\beta}:G_\alpha\to G_\beta \text{ for }\beta\le \alpha\}$
be a family of isomorphisms satisfying the following.
\begin{enumerate}[left=\parindent,label={\bf  (SDL\arabic*)}]
\item $\phi_{\alpha,\alpha}$ is the identity map on $G_\alpha$.
\item If $\gamma\le\beta\le\alpha$ in $D$ then
$\phi_{\alpha ,\beta}\phi_{\beta,\gamma}=\phi_{\alpha,\gamma}$.
\end{enumerate} 
Then  
\[S=\cup_{\alpha\in D }G_\alpha \]
is a semiring with product and sum defined as follows. For 
$x\in G_\alpha$ and $y\in G_\beta$
\begin{equation}
    x\cdot y=x\phi_{\alpha ,\alpha\beta}\cdot y\phi_{\beta ,\alpha\beta}
\end{equation}
 and  
 \begin{equation}
     x+y=x\psi_{\alpha ,\alpha +\beta}
 \end{equation}
where $\psi_{\beta,\alpha}=(\phi_{\alpha,\beta})^{-1}$ for $\beta\le\alpha$.
\end{thm}

Before starting the proof of the theorem we state a result which will be used in the proof.
\begin{lem}\label{compatibility}
If $\alpha+\beta\le\delta$ and $\gamma\le\alpha\beta$ then
$\phi_{\alpha ,\gamma}\psi_{\gamma ,\beta}=\psi_{\alpha ,\delta}\phi_{\delta ,\beta}$.
\end{lem}
The proof is as in the proof of statement \ref{stl6} of Theorem \ref{stronglatticethm}. \qed

\begin{proof}[Proof (of Theorem \ref{sdl})]
From Definition \ref{strongsldefn}  $(S,\cdot)$ is a  semigroup.
In a similar way we see that  $(S,+)$ is also a semigroup.
It remains to prove the distributivity.
Consider $x,y,z\in S$ where $x\in G_\alpha$, $y\in G_\beta$ and  $z\in G_\sigma$.
Let $\delta =\alpha +\beta$ and $\gamma=\delta\sigma$. Then 
$x+y\in G_\delta$ and $(x+y)z\in G_\gamma$. Now
 \begin{align*}
 (x+y)z &= (x+y)\phi_{\delta,\gamma}z\phi_{\sigma,\gamma}\\
  &=(x\psi_{\alpha,\delta})\phi_{\delta,\gamma}z\phi_{\sigma,\gamma}\\
   &=x(\phi_{\alpha,\alpha\gamma}\psi_{\alpha\gamma,\gamma})z\phi_{\sigma,\gamma}
   &\text{ by Lemma \ref{compatibility}}\\
   &=x\phi_{\alpha,\alpha\sigma}\psi_{\alpha\sigma,\gamma}z\phi_{\sigma,\gamma}
   &\text{ since }\alpha\gamma=\alpha\sigma .
 \end{align*}
Also $xz\in G_{\alpha\sigma} ,\ yz\in G_{\beta\sigma}$ and 
$xz+yz\in G_{\gamma}$ since 
$\alpha\sigma +\beta\sigma =(\alpha +\beta)\sigma =\delta\sigma =\gamma$. 
Then
 \begin{align*} 
 xz+yz &=x\phi_{\alpha,\alpha\sigma}z\phi_{\sigma,\alpha\sigma}+
 y\phi_{\beta,\beta\sigma}z\phi_{\sigma,\beta\sigma}\\
 &=(x\phi_{\alpha,\alpha\sigma}z\phi_{\sigma,\alpha\sigma})\psi_{\alpha\sigma,\gamma}\\
&=x\phi_{\alpha,\alpha\sigma}\psi_{\alpha\sigma,\gamma}z\phi_{\sigma,\alpha\sigma}\psi_{\alpha\sigma,\gamma} \text{ since }\psi\text{ is a homomorphism.}\\
 &=  x\phi_{\alpha,\alpha\sigma}\psi_{\alpha\sigma,\gamma}z\phi_{\sigma,\gamma}
 \text{ since } \alpha\sigma=\alpha\gamma\le \gamma\le \sigma\text{ and by Lemma \ref{compatibility}.} 
  \end{align*}
It follows that
$(x+y)z=xz+yz$.
Similarly we can prove that
$z(x+y)=zx+zy$.
 \end{proof}

The following observations of semirings which are strong distributive lattices of group semirings are easy to verify.

\begin{thm}
Let $D$ be a distributive lattice and   
$$S=\cup_{\alpha\in D }G_\alpha $$
be a distributive lattice of group semirings where each $G_\alpha$ is a group semiring.
Let $E(S)$ denote the set of all multiplicative idempotents of $S$. Then
\begin{enumerate}[label=\emph{(\roman*)}]
\item $E(S)=\{e_\alpha:\alpha\in D \}$
where $e_\alpha$ is the identity of the group $G_\alpha$.
\item $E(S)$ is a subsemiring of $S$ and $(E(S),+,\cdot)$ is a 
distributive lattice isomorphic to $D$.
\item The semigroup $(S,\cdot)$ is a Clifford semigroup and
\item  The semigroup $(S,+)$ is a left normal band\cite{yamada1958note}.
\end{enumerate}
\end{thm}

\bibliographystyle{acm}
\bibliography{semiring-books}

\end{document}